\RequirePackage{fix-cm}
\documentclass[smallextended]{svjour3}       
\smartqed  
\usepackage{graphicx}
\usepackage{algorithm}
\usepackage{algorithmic}
\usepackage{amsmath}
\usepackage{amsfonts}
\usepackage{amssymb}
\usepackage{url}
\usepackage{enumerate}
\usepackage[centering]{geometry}%

\providecommand{\U}[1]{\protect\rule{.1in}{.1in}}

\journalname{Preprint}
\begin{document}

\title{ParNes: A rapidly convergent algorithm for accurate recovery of sparse and
approximately sparse signals
}

\titlerunning{ParNes: A rapidly convergent algorithm for recovery of sparse signals}        

\author{Ming Gu         \and
        Lek-Heng Lim    \and
        Cinna Julie Wu
}


\institute{M. Gu \at
              Department of Mathematics, University of California at Berkeley, Berkeley, CA 94720-3840, USA\\
              Tel.: 510-642-3145\\
              \email{mgu@math.berkeley.edu}           
           \and
           L.-H. Lim \at
							Department of Statistics, University of Chicago, Chicago, IL 60637-1514, USA\\
							\email{lekheng@galton.uchicago.edu}
					 \and
           C. J. Wu \at
              Department of Mathematics, University of California at Berkeley, Berkeley, CA 94720-3840, USA\\
              \email{cinnawu@math.berkeley.edu}           
}

\date{Received: date / Accepted: date}

\maketitle

\begin{abstract}
In this article, we propose an algorithm, \textsc{nesta}-\textsc{lasso}, for the \textsc{lasso} problem, i.e., an underdetermined linear least-squares problem with a $1$-norm constraint on the solution. We prove under the assumption of the \textit{restricted isometry property} (\textsc{rip}) and a sparsity condition on the solution, that \textsc{nesta}-\textsc{lasso} is guaranteed to be almost always locally linearly convergent. As in the case of the algorithm \textsc{nesta} proposed by Becker, Bobin, and Cand\`{e}s, we rely on Nesterov's accelerated proximal gradient method, which takes $O(\sqrt{1/\varepsilon})$ iterations to come within $\varepsilon > 0$ of the optimal value. We introduce a modification to Nesterov's method that regularly updates the prox-center in a provably optimal manner, and the aforementioned linear convergence is in part due to this modification.

In the second part of this article, we attempt to solve the basis pursuit denoising (\textsc{bpdn}) problem (i.e., approximating the minimum $1$-norm solution to an underdetermined least squares problem) by using \textsc{nesta}-\textsc{lasso} in conjunction with the Pareto root-finding method employed by van den Berg and Friedlander in their \textsc{spgl1} solver. The resulting algorithm is called \textsc{parnes}. We provide numerical evidence to show that it is comparable to currently available solvers.  

\end{abstract}

\section{Introduction}
\label{sec:1}
We would like to find a solution to the sparsest recovery problem with noise
\begin{equation}
\min\; \lVert x\rVert_{0}\quad\text{s.t.}\quad\lVert Ax-b \rVert_{2}
\leq\sigma. \label{eqnZero}%
\end{equation}
Here, $\sigma$ specifies the noise level, $A$ is an $m$-by-$n$ matrix with $m \ll n$, and $\lVert x\rVert_{0}$ is the number of nonzero entries of $x$. This problem comes up in fields such as image processing \cite{imaging}, seismics
\cite{Seismic1,Seismic2}, astronomy \cite{astro}, and model selection in
regression \cite{LARegression}. Since \eqref{eqnZero} is known to be ill-posed and NP-hard \cite{NP1,NP2}, various convex, $l_{1}$-relaxed formulations are often used.

Relaxing the $0$-norm in \eqref{eqnZero} gives the basis pursuit
denoising (\textsc{bpdn}) problem
\begin{equation}
\textsc{bp}(\sigma)\quad\min\; \lVert x\rVert_{1} \quad\text{s.t.}\quad\lVert
Ax-b\rVert_{2}\leq\sigma. \label{bpsigma}%
\end{equation}
The special case of $\sigma= 0 $ is the basis pursuit problem \cite{BP}.
Two other commonly used $l_{1}$-relaxations are the \textsc{lasso} problem \cite{LS} 
\begin{equation}
\textsc{ls}(\tau)\quad\min\; \lVert Ax-b\rVert_{2} \quad\text{s.t.}\quad\lVert
x\rVert_{1}\leq\tau \label{lasso}%
\end{equation}
and the penalized least-squares problem
\begin{equation}
\textsc{qp}(\lambda)\quad\min\; \lVert Ax-b\rVert_{2}^{2}+\lambda\lVert
x\rVert_{1}%
\end{equation}
proposed by Chen, Donoho, and Saunders \cite{BP}. A large amount of work has been done to show that these formulations give an effective approximation of the solution to \eqref{eqnZero}; see \cite{RIP2,RIP3,RIP1}. In fact, under certain conditions on the sparsity of the solution to \eqref{eqnZero}, these formulations can exactly recover the solution, provided that $A$ satisfies the \textit{restricted isometry property} (\textsc{rip}). 

There is a wide variety of algorithms which solve the
\textsc{bp}$(\sigma)$, \textsc{qp}$(\lambda)$, and \textsc{ls}$(\tau)$
problems. Refer to Section~\ref{sec:SOLVERS} for descriptions of some of the current algorithms. Our work has been motivated by the accuracy and speed of the recent solvers \textsc{nesta} and \textsc{spgl1}. In \cite{Nes2}, Nesterov presents an algorithm to minimize a smooth convex function over a convex set with an optimal convergence
rate. An extension to the nonsmooth case is presented in \cite{Nes1}. \textsc{nesta} solves the \textsc{bp}$(\sigma)$ problem using the nonsmooth version of Nesterov's work. 

For appropriate parameter choices of $\sigma, \lambda,$ and
$\tau$, the solutions of \textsc{bp}$(\sigma)$, \textsc{qp}$(\lambda)$, and
\textsc{ls}$(\tau)$ coincide \cite{SPGL1}. Although the exact dependence is usually hard
to compute \cite{SPGL1}, there are solution
methods which exploit these relationships. The \textsc{matlab} solver \textsc{spgl1} is based on the Pareto root-finding method  \cite{SPGL1} which solves \textsc{bp}$(\sigma)$ by approximately solving a sequence of \textsc{ls}$(\tau)$ problems. In \textsc{spgl1}, the \textsc{ls}$(\tau)$ problems are solved using a spectral projected-gradient (\textsc{spg}) method. 

While we are ultimately interested in solving the \textsc{bpdn} problem in \eqref{bpsigma}, our main result is an algorithm for solving the \textsc{lasso} problem \eqref{lasso}. Our algorithm, \textsc{nesta}-\textsc{lasso} (cf.\ Algorithm~\ref{alg:nesta-lasso}), essentially uses Nesterov's work to solve the \textsc{lasso} problem. We introduce one improvement to Nesterov's original method, namely, we update the prox-center every $K$ steps instead of fixing it throughout the algorithm. With this modification, we prove in Theorem~\ref{thm:OPT} that \textsc{nesta}-\textsc{lasso} is guaranteed to be almost always locally linearly convergent for sufficiently large $K$, as long as the solution is $s$-sparse and $A$ satisfies the restricted isometry property of order $2s$. In fact, Theorem~\ref{thm:OPT} also provides the choice for the optimal $K$. 

Finally, we show that replacing the \textsc{spg} method in the Pareto root-finding procedure, used in \textsc{spgl1}, with our \textsc{nesta}-\textsc{lasso} method leads to an effective method for solving \textsc{bp}$(\sigma)$. We call this modification \textsc{parnes} and compare its efficacy with the state-of-the-art solvers presented in Section~\ref{sec:SOLVERS}.
\subsection{Notation and terminology}
In this paper, a vector is $s$-sparse if it has exactly $s$ nonzero elements. We say that a vector is \textit{at least} $s$-sparse if it has at most $s$ nonzero elements. For a nonzero, $s$-sparse vector $x\in\mathbb{R}^n$, let $I_x$ be the set of indices of the nonzero coefficients of $x$, i.e. the support of $x$; $\overline{x}$ is the vector containing the nonzero elements of $x$. For an $I\subseteq\left\{1,\ldots,n\right\}$, $I^c$ is the complement of $I$. Given a matrix $A\in\mathbb{R}^{m\times n}$ and $I\subseteq\left\{1,\ldots,n\right\}$, $A_I$ is the submatrix of $A$ containing the $j$-th columns of $A$ where $j\in I$. Throughout the paper, we use \textsc{matlab} terminology to describe vectors and matrices. Thus, $x[s:r]$ represents the subvector of $x$ containing elements $s$ to $r$. For a set $S$, let $\operatorname{int}(S)$ be the interior of $S$ and $\partial S$ be the boundary of $S$.   

\subsection{Organization of the paper}
\label{sec:1.1}
In Section~\ref{sec:NESTA-LASSO}, we present and describe the
background of \textsc{nesta}-\textsc{lasso}. We show in Section~\ref{sec:OPT} that, under some reasonable assumptions, \textsc{nesta}-\textsc{lasso} is almost always locally linear convergent. In Section~\ref{sec:PARNES}, we describe the Pareto root-finding procedure behind the \textsc{bpdn} solver \textsc{spgl1} and show how \textsc{nesta}-\textsc{lasso} can be used to solve a subproblem. Section 5 describes some of the available algorithms for solving \textsc{bpdn} and the equivalent \textsc{qp}$(\lambda)$ problem. Lastly, in Section 6, we show in a series of numerical experiments that using \textsc{nesta}-\textsc{lasso} in \textsc{spgl1} to solve \textsc{bpdn} is comparable with current competitive solvers.

\section{NESTA-LASSO\label{sec:NESTA-LASSO}}
We present the main parts of our method to solve the \textsc{lasso} problem. Our algorithm, \textsc{nesta}-\textsc{lasso} (cf.\ Algorithm~\ref{alg:nesta-lasso}), is an application of the accelerated
proximal gradient algorithm of Nesterov \cite{Nes2} outlined in Section~\ref{sec:NESTEROV}. Additionally, we have a prox-center update improving convergence which we describe in Section~\ref{sec:OPT}. In each iteration, we use the fast $l_1$-projector of Duchi et al. \cite{L1proj} given in Section~\ref{sec:1proj}.

\subsection{Nesterov's algorithm\label{sec:NESTEROV}}

Let $Q\subseteq\mathbb{R}^{n}$ be a convex closed set. Let $f:Q\rightarrow
\mathbb{R}$ be smooth, convex and, Lipschitz differentiable with $L$ as the Lipschitz
constant of its gradient, i.e.%
\begin{equation*}
\lVert\nabla f(x)-\nabla f(y)\rVert_2\leq L\rVert x-y\rVert_2,\quad\text{for all
}x,y\in Q.
\end{equation*}
Nesterov's accelerated proximal gradient algorithm iteratively defines a
sequence $x_{k}$ as a judiciously chosen convex combination of two other
sequences $y_{k}$ and $z_{k}$, which are in turn solutions to two quadratic
optimization problems on $Q$. The sequence $z_{k}$ involves a strongly convex
\textit{prox-function}, $d(x)$, which satisfies
\begin{equation}
d(x)\geq\frac{\alpha}{2}\lVert x-c\rVert_{2}^{2}. \label{eq:proxfn}%
\end{equation}
For simplicity, we have chosen the right-hand side of \eqref{eq:proxfn} with
$\alpha=1$ as our prox-function throughout this paper. The $c$ in the
prox-function is called the \textit{prox-center}. With this prox-function, we have:

\begin{align*}
y_{k}  &  =\operatorname*{argmin}_{y\in Q}\nabla f(x_{k})^{\top}%
(y-x_{k})+\frac{L}{2}\lVert y-x_{k}\rVert_{2}^{2},\\
z_{k}  &  =\operatorname*{argmin}_{z\in Q}\sum_{i=0}^{k}\frac{i+1}{2}%
[f(x_{i})+\nabla f(x_{i})^{\top}(z-x_{i})]+\frac{L}{2}\lVert z-c\rVert_{2}%
^{2},\\
x_{k}  &  =\frac{2}{k+3}z_{k}+\frac{k+1}{k+3}y_{k}.
\end{align*}

Nesterov showed that if $x^{\ast}$ is the optimal solution to
\[
\min_{x\in Q}f(x),
\]
then the iterates defined above satisfy%
\begin{equation*}
f(y_{k})-f(x^{\ast})\leq\frac{L}{k(k+1)}\lVert x^{\ast}-c\rVert_{2}%
^{2}=O\left(  \frac{L}{k^{2}}\right)  .
\end{equation*}
An implication is that the algorithm requires $O(\sqrt{L/\varepsilon})$
iterations to bring $f(y_{k})$ to within $\varepsilon>0$ of the optimal value.

\begin{algorithm}[htbp]
\caption{Accelerated proximal gradient method for convex minimization}
\label{alg:nesta}
\begin{algorithmic}[1]
\REQUIRE{function $f$, gradient $\nabla f$, Lipschitz constant $L$, prox-center $c$.}
\ENSURE{$x^* = \operatorname*{argmin}_{x\in Q} f(x)$ }
\STATE initialize $x_{0}$;
\FOR{$k=0,1,2,\dots$,}
\STATE compute $f(x_{k})$ and $\nabla f(x_{k})$;
\STATE $y_{k}=\operatorname*{argmin}_{y\in Q}\nabla f(x_{k})^{\top}
(y-x_{k})+\frac{L}{2}\lVert y-x_{k}\rVert_{2}^{2}$;
\STATE $z_{k}=\operatorname*{argmin}_{z\in Q}\sum_{i=0}^{k}\frac{i+1}{2}
[f(x_{i})+\nabla f(x_{i})^{\top}(z-x_{i})]+\frac{L}{2}\lVert z-c\rVert_{2}^{2}$;
\STATE $x_{k}=\frac{2}{k+3}z_{k}+\frac{k+1}{k+3}y_{k}$;
\ENDFOR
\end{algorithmic}
\end{algorithm}

In \cite{Nes1}, Nesterov extends his work to minimize nonsmooth convex functions $f$. Nesterov shows that one can obtain the minimum by applying his algorithm for smooth minimization to a smooth approximation $f_{\mu}$ of $f$. Since $\nabla f_{\mu}$ is shown to have Lipschitz constant $L_{\mu} = 1/\mu$, if $\mu$ is chosen to be proportional to $\varepsilon$, it takes $O\left(\frac{1}{\varepsilon}\right)$ iterations to bring $f(x_k)$ within $\varepsilon$ of the optimal value. 

The recent algorithm \textsc{nesta} solves \textsc{bp}$({\sigma})$ using Nesterov's algorithm for nonsmooth minimization. Our algorithm, \textsc{nesta}-\textsc{lasso}, solves \textsc{ls}$({\tau})$ using Nesterov's smooth minimization algorithm. We are motivated by the accuracy and speed of \textsc{nesta} and the fact that the smooth version of Nesterov's algorithm has a faster convergence rate than the nonsmooth version.

\subsection{NESTA-LASSO-K: An accelerated proximal gradient algorithm for LASSO}
\label{sec:NESTA1}

We apply Nesterov's accelerated proximal gradient method, Algorithm~\ref{alg:nesta}, to the \textsc{lasso} problem  \textsc{ls}$(\tau)$. We make one slight improvement to Algorithm~\ref{alg:nesta}. Namely, we update our prox-centers every $K$ steps (cf.\ Algorithm~\ref{alg:nesta-update}); that is, Algorithm~\ref{alg:nesta} is restarted every $K$ iterations with a new prox-center. We will see that this leads to local linear convergence under a suitable application of \textsc{rip} (see Corollary~\ref{cor:OPT} for details). In fact, we show in Section~\ref{sec:OPT} that the prox-centers may be updated in an optimal fashion (cf.\ Algorithm~\ref{alg:nesta-lasso}). 


In our case, $f=\frac{1}{2}\lVert b - Ax \rVert_{2}^{2}$, $\nabla f = A^{\top
}(Ax-b)$, and $Q$ is the $1$-norm ball $\lVert x \rVert_{1}\leq\tau$. The
initial point $x_{0}$ is used as the prox-center $c$. To compute the iterate
$y_{k}$, we have
\begin{align*}
y_{k}  &  =\operatorname*{argmin}_{\lVert y\rVert_{1}\leq\tau}\nabla
f(x_{k})^{\top} (y-x_{k})+\frac{L}{2}\lVert y-x_{k}\rVert_{2}^{2}\\
&  =\operatorname*{argmin}_{\lVert y\rVert_{1}\leq\tau}y^{\top}y - 2(x_{k}
-\nabla f(x_{k})/L)^{\top}y\\
&  =\operatorname*{argmin}_{\lVert y\rVert_{1}\leq\tau}\lVert y-(x_{k}-\nabla
f(x_{k})/L)\rVert_{2}\\
&  =\operatorname{proj}_{1}(x_{k}-\nabla f(x_{k})/L,\tau)
\end{align*}
where $\operatorname{proj}_{1}(v,\tau)$ returns the projection of the vector $v$ onto
the $1$-norm ball of radius $\tau$. By similar reasoning, computing $z_{k}$
can be shown to be equivalent to computing
\[
z_{k}=\operatorname{proj}_{1}\left(c-\frac{1}{L}\sum\nolimits_{i=0}^{k}{\frac{i+1}{2}
\nabla f(x_{i})},\tau\right).
\]
In each iteration, we use the fast $l_1$-projector $\operatorname{proj}_{1}$ described in the next section.

In \textsc{nesta}-\textsc{lasso}-\textsc{k}, Nesterov's method is restarted every $K$ steps with the new prox-center $\operatorname{proj}_{1}(y_{iK}-\nabla f(y_{iK})/L,\tau)$. Here, $y_{iK}$ is the $K$-th iterate of Nesterov's method after the $i$-th prox-center change; see Algorithm~\ref{alg:nesta-update}. In \textsc{nesta}-\textsc{lasso}, Nesterov's method is restarted in the same manner, except $K$ is chosen in an optimal way. 

\begin{algorithm}[htbp]
\caption{\textsc{nesta}-\textsc{lasso}-\textsc{k} algorithm with prox-center updates every $K$ steps}
\label{alg:nesta-update}
\begin{algorithmic}[1]
\REQUIRE{initial point $x_0$, \textsc{lasso} parameter $\tau$, tolerance $\eta$, steps to update $K$}
\ENSURE{$x_{\tau} = \operatorname*{argmin} \{ \lVert b - Ax \rVert_2 : \lVert x \rVert_1 \le \tau \}$. }
\FOR{$j=0,\dots,j_{\max}$,}
\STATE $c_j = x_0$, $h_{0}=0$, $r_{0}=b-Ax_{0}$, $g_{0}=-A^{\top}r_{0}$, $\eta_{0}=\lVert r_{0}\rVert_{2}-(b^{\top}r_0-\tau\lVert g_{0}\rVert_{\infty})/\lVert r_{0}\rVert_{2}$;
\FOR{$k=0,\dots,K$}
\STATE $y_{k}=\operatorname{proj}_{1}(x_{k}-g_{k}/L,\tau)$;
\STATE $h_{k}=h_{k}+\frac{k+1}{2}g_{k}$;
\STATE $z_{k}=\operatorname{proj}_{1}(c_{j}-h_{k}/L,\tau)$;
\STATE $x_{k}=\frac{2}{k+3}z_{k}+\frac{k+1}{k+3}y_{k}$;
\STATE $r_{k}=b-Ax_{k}$;
\STATE $g_{k}=-A^{\top}r_{k}$;
\STATE $\eta_{k}=\lVert r_{k}\rVert_{2}-(b^{\top}r_k-\tau\lVert g_{k}\rVert_{\infty})/\lVert r_{k}\rVert_{2}$;
\ENDFOR
\STATE $x_0 = \operatorname{proj}_{1}(y_{k}+A^{\top}(b-Ay_{k})/L,\tau)$;
\IF{$\eta_{k} \leq \eta$}
\RETURN $x_{\tau} = y_{k}$;
\ENDIF
\ENDFOR
\end{algorithmic}
\end{algorithm}

\begin{algorithm}[htbp]
\caption{\textsc{nesta}-\textsc{lasso} algorithm with optimal prox-center updates}
\label{alg:nesta-lasso}
\begin{algorithmic}[1]
\REQUIRE{initial point $x_0$, \textsc{lasso} parameter $\tau$, tolerance $\eta$.}
\ENSURE{$x_{\tau} = \operatorname*{argmin} \{ \lVert b - Ax \rVert_2 : \lVert x \rVert_1 \le \tau \}$. }
\FOR{$j=0,\dots,j_{\max}$,}
\STATE $c_j = x_0$, $h_{0}=0$, $r_{0}=b-Ax_{0}$, $g_{0}=-A^{\top}r_{0}$, $\eta_{0}=\lVert r_{0}\rVert_{2}-(b^{\top}r_0-\tau\lVert g_{0}\rVert_{\infty})/\lVert r_{0}\rVert_{2}$;
\FOR{$k=0,\dots,k_{\max}$,}
\IF{$\eta_{k} \leq e^{-2}\eta_0$}
\RETURN $y_{k}, \eta_{k}$
\ENDIF
\STATE $y_{k}=\operatorname{proj}_{1}(x_{k}-g_{k}/L,\tau)$;
\STATE $h_{k}=h_{k}+\frac{k+1}{2}g_{k}$;
\STATE $z_{k}=\operatorname{proj}_{1}(c_{j}-h_{k}/L,\tau)$;
\STATE $x_{k}=\frac{2}{k+3}z_{k}+\frac{k+1}{k+3}y_{k}$;
\STATE $r_{k}=b-Ax_{k}$;
\STATE $g_{k}=-A^{\top}r_{k}$;
\STATE $\eta_{k}=\lVert r_{k}\rVert_{2}-(b^{\top}r_k-\tau\lVert g_{k}\rVert_{\infty})/\lVert r_{k}\rVert_{2}$;
\ENDFOR
\STATE $x_0 = \operatorname{proj}_{1}(y_{k}+A^{\top}(b-Ay_{k})/L,\tau)$; 
\IF{$\eta_{k} \leq \eta$}
\RETURN $x_{\tau} = y_{k}$;
\ENDIF
\ENDFOR
\end{algorithmic}
\end{algorithm}


\subsection{$l_1$-projector\label{sec:1proj}}
The projection of an $n$-vector, $d$, onto the $1$-norm ball, $\lVert
x\rVert_{1}\leq\tau$, is the solution to the minimization problem
\begin{equation*}
\operatorname{proj}_{1}(d,\tau):= \operatorname*{argmin}_{x} \lVert d - x\rVert
_{2}\quad\text{s.t.}\quad\lVert x\rVert_{1}\leq\tau.
\end{equation*}
Let $\overline{d}$ be a reordering of $d$ with $|\overline{d}_1| \geq ...\geq |\overline{d}_n|$. Then $a =\operatorname{proj}_{1}(d,\tau)$, is given by
\begin{equation}\label{proj1}
a_{i} = \operatorname{sgn}(d_{i})\cdot\max\{0,|d_{i}| - \eta\}\quad\text{with}\quad\eta= \frac{(|\overline{d}_{1}|+\dots+|\overline{d}_{k}|)-\tau}{k},%
\end{equation}
where $k$ is the largest index such that $\eta\leq |\overline{d}_{k}|$. 

See \cite{L1proj}, by Duchi et al., and \cite{SPGL1} for fast algorithms to compute $a$. Such algorithms cost $O(n\log{n})$ in the worst case but have been shown experimentally to cost much less \cite{SPGL1}. The results in \cite{Tseng} imply the two calls to $\operatorname{proj}_{1}$ in the inner loop of \textsc{nesta}-\textsc{lasso} can be reduced to one call, but due to the low cost of $\operatorname{proj}_{1}$, we do not make this modification.

\section{Local linear convergence and optimality\label{sec:OPT}}

Under reasonable assumptions on the matrix $A$ and the solution $x^{\ast}$ of the \textsc{lasso} problem, we prove that \textsc{nesta}-\textsc{lasso}-\textsc{k} almost always has a local linear convergence rate for large enough $K$. We also show that we can update the prox-centers $c$ in a provably optimal way (\textsc{nesta}-\textsc{lasso}). Let $y_k$ be the $k$-th iterate of Nesterov's accelerated proximal gradient method when minimizing a function $f$. Recall, 
\begin{equation}
f(y_{k})-f(x^{\ast})\leq\frac{L}{k(k+1)}\lVert x^{\ast}-c\rVert_{2}%
^{2} \label{eq:Lip}%
\end{equation}
where $L$ is the Lipschitz constant for $\nabla f$ and $c$ is the prox-center \cite{Nes2,Nes1}. 

In our case, $f(x)=\frac{1}{2}\lVert Ax-b\rVert_{2}^{2}$, where $A$ is a wide matrix. We will assume that $A$ satisfies the \textit{restricted isometry
property} $(\textsc{rip})$ of order $2s$ as described in \cite{DS,RIP}. Namely, there exists a
constant $\delta_{2s}\in(0,1)$ such that%
\begin{equation}
(1-\delta_{2s})\lVert x\rVert_{2}^{2}\leq\lVert Ax\rVert_{2}^{2}\leq
(1+\delta_{2s})\lVert x\rVert_{2}^{2} \label{eq:RIP}%
\end{equation}
whenever $\lVert x\rVert_{0}\leq2s$. Since the \textsc{rip} helps ensure that the solution to
\eqref{eqnZero} is closely approximated by the solution to \eqref{bpsigma} \cite{RIP}, and we are ultimately interested in solving \eqref{bpsigma}, this is a reasonable assumption. Moreover, since we hope to recover the sparse solution to the solution to \eqref{eqnZero}, we assume that the solution $x^*$ to the \textsc{lasso} problem is $s$-sparse. We plan to analyze the approximately sparse case for future work. 

It turns out that under these assumptions, the sequence of $y_k$'s converges to the solution $x^*$. 

\begin{lemma}\label{yconverge}If $A$ satisfies the \textit{restricted isometry property} $(\textsc{rip})$ of order $2s$, and the optimal solution $x^*$ is $s$-sparse, then the sequence of $y_k$'s converges to $x^*$.
\end{lemma}
\begin{proof} Under the \textsc{rip} and the assumption that $x^{*}$ is $s$-sparse, it follows by Theorem~5 in \cite{OPT} that the \textsc{lasso} problem has a unique solution. Since the $y_k$'s lie on the $1$-norm ball, and the $1$-norm ball is compact, this implies that the sequence of $y_{k}$'s must converge to $x^{*}$.\qed\end{proof}
 
\subsection{Almost sure sparsity of Nesterov's method}

We first state and prove the following results before proving our main results, i.e. the local linear convergence of  \textsc{nesta}-\textsc{lasso}-\textsc{k} and the optimality of \textsc{nesta}-\textsc{lasso}. In particular, we show that under certain assumptions on the \textsc{lasso} problem, the solution is almost always non-degenerate (see Proposition \ref{lemma:opt1}), and the iterates of Algorithm \ref{alg:nesta} are almost always eventually $s$-sparse. Our first lemma describes when the image of proj$_1$ is $s$-sparse.

For $d\in\mathbb{R}^n$ with $|d_1|\geq\ldots\geq |d_n|$, recall from Section \ref{sec:1proj} that $a =\operatorname{proj}_{1}(d,\tau)$ is given by
\begin{equation}\label{proj11}
a_i = \max\{0,|d_{i}| - \eta\}\quad\text{with}\quad\eta= \frac{|d_{1}|+\dots+|d_{k}|-\tau}{k},%
\end{equation}
where $k$ is the largest index such that $\eta\leq |d_{k}|$. For each $i \in \{1,\ldots,n\}$, define
\begin{equation*}
\eta_i := \frac{|d_{1}|+\dots+|d_{i}|-\tau}{i}.
\end{equation*} 
The $\eta_i$'s satisfy the following property which is used in the proof of our first lemma.

\begin{claim}\label{claim:eta}
$\eta = \max\left\{\eta_i: i = 1,\ldots,n\right\}$.
\end{claim}
\begin{proof}
A simple algebraic manipulation shows that $\eta_i - \eta_{i-1} = \frac{1}{i-1}(d_i - \eta_i)$ for $i \in \{2,\ldots,n\}$. Thus, $\operatorname{sgn}(\eta_i - \eta_{i-1}) = \operatorname{sgn}(d_i - \eta_i)$. Suppose $\eta = \eta_k$ for some $k$. Then $\eta_k \leq d_k$. Since $\operatorname{sgn}(\eta_i - \eta_{i-1}) = \operatorname{sgn}(d_i - \eta_i)$, it follows that $\eta_{k-1}\leq\eta_{k}$ and so $\eta_{k-1} \leq d_{k-1}$; thus, we can repeatedly apply this argument to show that $\eta_i\leq\eta_k$ for any $i<k$. A similar argument shows that $\eta_i\leq\eta_k$ for any $i>k$.
\qed\end{proof}

Given a nonempty $I\subseteq\left\{1,\ldots,n\right\}$ with $|I| = s$ and $\tau > 0$, if $s<n$, define the set
\begin{equation*}
C_{I,\tau} := \left\{x\in\mathbb{R}^n : \sum\nolimits_{i\in I} |x_i| - \tau \geq s\cdot|x_j| \text{ for } j\notin I\right\}.
\end{equation*}
If $I = \left\{1,\ldots,n\right\}$, let $C_{I,\tau} := \left\{x\in\mathbb{R}^n : \|x\|_1 \geq\tau \right\}$. Note that $C_{I,\tau}$ is a union of cones. The following lemma shows that proj$_1$ sends vectors in $C_{I,\tau}$ to vectors that are \textit{at least} $s$-sparse. 

\begin{lemma}\label{sparse1} If $d\in C_{I,\tau}$ then $I_{\operatorname{proj}_{1}(d,\tau)}\subseteq I$. Namely, $\operatorname{proj}_{1}(d,\tau)$ is at least $s$-sparse. 
\end{lemma}
\begin{proof} 
Suppose $d\in C_{I,\tau}$ with $d\geq 0$. For simplicity, assume that $I = \left\{1,\ldots,s\right\}$, $d_1\geq\ldots\geq d_s$, and $d_{s+1}\geq\ldots\geq d_n$. The proof is easily generalized to all other cases.  
	
By $(\ref{proj11})$, $a_{s+1}\geq\ldots\geq a_n$, so it is enough to show that $a_{s+1} = 0$. Since $d\in C_{I,\tau}$, 
\begin{equation}\label{d2}
	 s\cdot d_{s+1} \leq d_1+\ldots+d_s - \tau.
\end{equation}
	Let $r\leq s$ be the largest index such that $d_r\geq d_{s+1}$. Such an $r$ exists since  $s\cdot d_1 \geq d_1+\dots+d_s - \tau\geq s\cdot d_{s+1}$. By $(\ref{d2})$,
	\begin{eqnarray*}	
	  r\cdot d_{s+1} &\leq& d_1 +\dots+d_r + (d_{r+1}-d_{s+1}) + \dots +(d_{s}-d_{s+1}) -\tau\\ 
	  &\leq& d_1 +\dots+d_r -\tau,  
\end{eqnarray*}
which implies,
\begin{equation*}
	d_{s+1} \leq \frac{d_1+\dots+d_r+d_{s+1}-\tau}{r+1}.
		\end{equation*}
	By the above claim, $d_{s+1}\leq\eta$, and so $a_{s+1} = 0$.\qed

\end{proof}

The next few lemmas involve the \textsc{lasso} problem. First note the following \textsc{lasso} optimality conditions (see e.g.\ \cite{Pathwise} and \cite{LassoOpt}). 
\begin{proposition}[LASSO optimality conditions] For an $x^*\in\mathbb{R}^n$, let $I = I_{x^*}$. Then $x^*$ is the optimal solution to the \textsc{lasso} problem if and only if the gradient, $-\nabla f(x^*) = A^{\top}(b-Ax^*)$, at $x^*$ satisfies 
\begin{align}\label{lemma:opt1}
A_{I}^{\top}(b-A_{I}\overline{x}^*) = \gamma\cdot \operatorname{sgn}(\overline{x}^*),\\
\lVert A_{I^c}^{\top}(b-A_{I}\overline{x}^*)\rVert_\infty \leq \gamma.\label{opt2}
\end{align}
for some $\gamma\geq0$. Moreover, there is a one-to-one correspondence between the $\gamma$ and $\tau$. Following the typical convention, if \eqref{opt2} is a strict inequality, we say that $x^*$ is a non-degenerate solution. Otherwise, we say that $x^*$ is a degenerate solution. 
\end{proposition}
%

The following lemma relates non-deg
\begin{lemma}\label{lemma:int}If $x^*$ is a non-degenerate solution with $I_{x^*} = I$, then $x^* - \nabla f(x^*)/L\in \operatorname{int}(C_{I,\tau})$.\end{lemma}
\begin{proof} By \eqref{lemma:opt1} and \eqref{opt2}, for any $j\in I$, we have 
\begin{align*}
\sum_{i\in I} \left|x_i + \frac{a_i^{\top}(b - A_I\overline{x}^*)}{L}\right|\quad-\quad\tau &= \sum_{i\in I} \left|x_i + \frac{\gamma\cdot\operatorname{sgn}(x_i)}{L}\right|\quad-\quad\tau \\
&= \sum_{i\in I} \left|x_i\right|\quad + \quad |I|\cdot\frac{\gamma}{L}\quad-\quad\tau \\
&\geq |I|\cdot|a_j^{\top}(b - A_I\overline{x}^*)|\\
&=|I|\cdot|x_j + a_j^{\top}(b - A_I\overline{x}^*)|.
\end{align*}
The third equation on the right holds since we must have $\|x^*\|_1 = \tau$. If not, then we must have $Ax^* - b = 0$ which is only possible when $x^*$ is a degenerate solution.
\qed\end{proof}

We now prove that under our assumptions on the \textsc{lasso} problem, the gradient at the optimal solution will almost always lie in a desirable direction. In other words, we have the following result.
  
\begin{theorem}\label{remark:direction}Suppose $A\in\mathbb{R}^{m\times n}$ satisfies the \textit{restricted isometry property} $(\textsc{rip})$ of order $2s$, and the optimal solution $x^*$ is $s$-sparse. The solution $x^*$ will almost always be non-degenerate.
\end{theorem}
\begin{proof}
Fix positive integers $m$, $n$, and $I\subseteq\left\{1,\ldots,n\right\}$ with $|I| = s \leq m$. Define $\textsc{ls}(m,n,I)$ to be the set of \textsc{lasso} problems 
\begin{equation*}
\min\; \lVert Ax-b\rVert_{2} \quad\text{s.t.}\quad\lVert
x\rVert_{1}\leq\tau
\end{equation*}
with $s$-sparse solutions $x^*$ such that $I_{x^*} = I$ and $A\in\mathbb{R}^{m\times n}$ satisfying the \textsc{rip} of order $2s$. As seen in the proof of Lemma~\ref{yconverge}, $x^*$ is unique.

The \textsc{lasso} optimality conditions above say that $x^*$ is the solution to a \textsc{lasso} problem if and only if $A_I^{\top}(b - A_I\overline{x}^*) = \gamma\cdot\operatorname{sgn}(\overline{x}^*)$ and $\lVert A_{I^c}^{\top}(b - A_I\overline{x}^*)\rVert_\infty \leq \gamma$ for some $\gamma\geq 0$. Since there is a one-to-one correspondence between $\tau$ and $\gamma$, we represent each \textsc{lasso} problem in $\textsc{ls}(m,n,I)$ with the quadruple $(A_I,A_{I^c},b,\gamma)$. Following this notation,
\begin{equation*}
\textsc{ls}(m,n,I) = T_1\cup T_2 
\end{equation*}
where
\begin{align*}
T_1 &:= \left\{(A_I,A_{I^c},b,\gamma) \in \textsc{ls}(m,n,I) : \|A_{I^c}^{\top}(b - A_1\overline{x}^*)\|_{\infty} = \gamma\right\},\\
T_2 &:= \left\{(A_I,A_{I^c},b,\gamma) \in \textsc{ls}(m,n,I) : \|A_{I^c}^{\top}(b - A_1\overline{x}^*)\|_{\infty} < \gamma\right\}.    
\end{align*}
We show that $T_1$ has Lebesgue measure zero and $T_2$ has nonzero Lebesgue measure.  

By the \textsc{rip}, $A_I$ has full rank since
\begin{equation*}
0<(1-\delta_{2s})\lVert x\rVert_{2}^{2}\leq\lVert A_Ix\rVert_{2}^{2}\leq
(1+\delta_{2s})\lVert x\rVert_{2}^{2} 
\end{equation*}
for all nonzero $x\in\mathbb{R}^s$. Thus, $A_I^{\top}A_I$ is invertible, and if $x^*$ is the solution to  $(A_I,A_{I^c},b,\gamma)\in\textsc{ls}(m,n,I)$ then 
\begin{equation*}
\overline{x}^* = (A_I^{\top}A_I)^{-1}(A_I^{\top}b - \gamma\cdot\operatorname{sgn}(\overline{x}^*)).
\end{equation*}

Let $U := \left\{(A_I,A_{I^c},b,\gamma)\in\mathbb{R}^{m\times s}\times \mathbb{R}^{m\times (n-s)}\times \mathbb{R}^{m}\times \mathbb{R}^+ :  A_I \text{ nonsingular}\right\}$. For each $w\in\left\{-1,1\right\}^s$, define the function $g_w : U \rightarrow \mathbb{R}^{n-s}$ by
\begin{equation*}
g_w(A_I,A_{I^c},b,\gamma) = \frac{A_{I^c}^{\top}\left(b - A_I(A_I^{\top}A_I)^{-1}(A_I^{\top}b - \gamma\cdot w)\right)}{\gamma},
\end{equation*}
If $S := \left\{x\in\mathbb{R}^{(n-s)} : |x| \leq 1\right\}$ with boundary $\partial S$ and interior $\operatorname{int}(S)$, then
\begin{equation*}
T_1 \subseteq \quad\bigcup_w g_w^{-1}(\partial S)\quad \bigcup \quad\mathbb{R}^{m\times s}\times \mathbb{R}^{m\times (n-s)}\times \mathbb{R}^{m}\times \left\{0\right\}.
\end{equation*}

Each component function of $g_w$ involves exactly one row of the variables in $A_{I^c}^{\top}$, and $g_w$ is the composition of matrix inversion and basic matrix operations. Thus, $g_w$ is a smooth map of constant rank $(n-s)$ on the open set $U\setminus g_w^{-1}(0)$. An application of Theorem 1 of \cite{measure0} shows that $g_w^{-1}(\partial S)$ has measure zero. Hence, $T_1$ has Lebesgue measure zero. 

To see that $T_2$ has nonzero measure, note that $T_2$ is the set of $(A_{I},A_{I^c},b,\gamma)\in U$ such that $A$ satisfies the \textsc{rip} of order $2s$ intersected with 
\begin{equation*}
\bigcup_w g_w^{-1}(\operatorname{int}(S))\cap \left\{(A_I,A_{I^c},b,\gamma)\in U:\operatorname{sgn}\left((A_{I}^{\top}A_I)^{-1}(A_I^{\top}b-\gamma\cdot w)\right) =w\right\}.  
\end{equation*}
Using the triangle inequality, it is easy to see that the former set is open since
\begin{equation*}
(1-\delta_{2s})\lVert x\rVert_{2}^{2}\leq\lVert Ax\rVert_{2}^{2}\leq
(1+\delta_{2s})\lVert x\rVert_{2}^{2} 
\end{equation*}
holds under small perturbations of $A$. The latter set is open since $g_w$  and $(A_I,A_{I^c},b,\gamma)\mapsto(A_I^{\top}A_I)^{-1}(A_I^{\top}b-\gamma\cdot w)$ are continuous functions for each $w$. Thus, $T_2$ is open. Moreover, it is easy to see that if $(A_I,A_{I^c},b,\gamma) \in T_1$ then there exists a small perturbation $E$ such that $(A_I,A_{I^c}+E,b,\gamma)\in T_2$. If $\textsc{ls}(m,n,I)$ is nonempty, it must be that $T_2$ is nonempty and therefore, has nonzero measure. 

This argument is easily extended for any $I\subseteq\left\{1,\ldots,n\right\}$. Since there are a finite number of $I$'s and a finite union of measure zero sets has measure zero, our lemma holds. 
\qed\end{proof}

Let $y_k$ be the $k$-th iterate of Nesterov's accelerated proximal gradient method applied to the \textsc{lasso} problem. The previous results allow us to make the following conclusion regarding the sparsity of $y_k$.

\begin{theorem}\label{theorem:sparse}Suppose $A$ satisfies the \textit{restricted isometry property} $(\textsc{rip})$ of order $2s$, and the optimal solution $x^*$ is $s$-sparse. The iterates $y_k$ are almost always eventually $s$-sparse.
\end{theorem}
\begin{proof}
By Lemma~\ref{yconverge}, the sequence $\{y_k\}$ converges to to the optimal solution $x^*$. Since $x_{k}=\frac{2}{k+3}z_{k}+\frac{k+1}{k+3}y_{k}$ and $\nabla f(x) = A^{\top}(Ax-b)$ is continuous, the sequence $\left\{x_k\ - \nabla f(x_k)/L\right\}$ converges to $x^* - \nabla f(x^*)/L$.

Theorem~\ref{remark:direction} says that $x^*$ is almost always non-degenerate, in which case, by Lemma~\ref{lemma:int}, $x^* - \nabla f(x^*)/L\in \operatorname{int}(C_{I{x^*},\tau})$, where $\operatorname{int}(C_{I{x^*},\tau})$ is the interior of $C_{I{x^*},\tau}$. Thus, if $x^*$ is non-degenerate, there exists an $N$ such that for $k \geq N$, $x_k - \nabla f(x_k)/L\in \operatorname{int}(C_{I_{x^*}})$. By Lemma~\ref{sparse1}, for such $k$, $y_k = \operatorname{proj}_{1}(x_{k}-\nabla f(x_{k})/L,\tau)$ is $s$-sparse.
\qed\end{proof}

\subsection{Local linear convergence of NESTA-LASSO}
We now show that \textsc{nesta-lasso-k}, Algorithm~\ref{alg:nesta-update}, is almost always locally linearly convergent under certain assumptions. First we give some motivation for why we update the prox-centers in \textsc{nesta-lasso-k}. 

Consider applying Nesterov's accelerated proximal gradient method, Algorithm~\ref{alg:nesta}, to the \textsc{lasso} problem. Suppose $A$ satisfies the \textit{restricted isometry property} $(\textsc{rip})$ of order $2s$ and the optimal solution $x^*$ is $s$-sparse. As seen in Theorem~\ref{theorem:sparse}, the iterates $y_k$ are almost always eventually $s$-sparse. Thus, it is reasonable to assume that $y_k$ is $s$-sparse.

Let $\delta=1-\delta_{2s}$ where $\delta_{2s}$ is the \textsc{rip} constant of $A$. We have%
\begin{equation}\label{eqn1}
\lVert A(x^{\ast}-y_{k})\rVert_{2}^{2}+2(y_{k}-x^{\ast})^{\top}A^{\top}(Ax^{\ast
}-b)=f(y_{k})-f(x^{\ast})\geq\lVert A(y_{k}-x^{\ast})\rVert_{2}^{2}\geq
\delta\lVert y_{k}-x^{\ast}\rVert_{2}^{2}.%
\end{equation}
To see the first inequality, let $y = x^{\ast}+\tau(y_k-x^{\ast})$ for $\tau\in[0,1]$. Due to the convexity of the $1$-norm ball, $y$ is feasible. Since $x^{\ast}$ is the minimum, for any $\tau\in[0,1]$,
\[
f(y)-f(x^{\ast}) = \tau^2\lVert A(x^{\ast}-y_{k})\rVert_{2}^{2}+2\tau(y_{k}-x^{\ast})^{\top}A^{\top}(Ax^{\ast
}-b)\geq 0.
\] 
Thus, $(y_{k}-x^{\ast})^{\top}A^{\top}(Ax^{\ast}-b)\geq 0$. The second inequality follows from \eqref{eq:RIP}.
Then from \eqref{eq:Lip}, we have%
\[
\delta\lVert y_{k}-x^{\ast}\rVert_{2}^{2}\leq\frac{L}{k(k+1)}\lVert x^{\ast
}-c\rVert_{2}^{2}.
\]
Putting everything together gives%
\begin{equation}
\lVert y_{k}-x^{\ast}\rVert_{2}\leq\sqrt{\frac{L}{k(k+1)\delta}}\lVert
x^{\ast}-c\rVert_{2}\leq\frac{1}{k}\sqrt{\frac{L}{\delta}}\lVert c-x^{\ast
}\rVert_{2}. \label{eq:decrease}%
\end{equation}
The above relation and \eqref{eq:Lip} suggest that when solving the \textsc{lasso} problem, we can speed up Algorithm~\ref{alg:nesta} by updating the prox-center, $c$, every $K$ steps. With our assumptions, we prove in the first part of following theorem that for every $K>\sqrt{\frac{L}{\delta}}$, restarting Algorithm~\ref{alg:nesta} every $K$ steps with the new prox-center, $\operatorname{proj}_{1}(y_{k}-\nabla f(y_k)/L,\tau)$, is locally linearly convergent. In the second part of Theorem~\ref{thm:OPT}, we prove that there is an optimal number of such steps. 

In the following, allow the iterates to be represented by $y_{jk}$ where $j$ is the number of times the prox-center has been changed (the outer iteration) and $k$ is number of iterations after the last
prox-center change (the inner iteration). Let the $j$-th prox-center update be represented by $p_j$.

\begin{theorem}\label{thm:OPT}Suppose $A$ satisfies the \textit{restricted isometry property} of order $2s$ and the solution $x^{\ast}$ is $s$-sparse. The following holds if $x^*$ is non-degenerate. 
\begin{enumerate}[\upshape (i)]
	\item Algorithm~\ref{alg:nesta-update} is locally linearly convergent for any $K>\sqrt{\frac{L}{\delta}}$.
	\item In Algorithm 2, let $j_{\operatorname*{tot}}$ be the total number of prox-center changes. The total number of iterations, $j_{\operatorname*{tot}}\cdot K$, to get $\lVert
p_{j}-x^{\ast}\rVert_{2}\leq\varepsilon$ is minimized if $K$ is equal to 
\begin{equation}
k_{\operatorname*{opt}}:=e\sqrt{\frac{L}{\delta}}\label{eq:kopt}%
\end{equation}
where $e$ is the base of the natural logarithm. Moreover, for each $j$, 
\[
\lVert p_{j}-x^{\ast}\rVert_2\leq\frac{1}{e^j}\lVert p_{0}%
-x^{\ast}\rVert_2.
\]
\end{enumerate}
\end{theorem}

\begin{proof}
\begin{enumerate}[\upshape (i)]
\item By Lemma~\ref{lemma:int}, $x^* - \nabla f(x^*)/L\in \operatorname{int}(C_{I_{x^*},\tau})$, where $\operatorname{int}(C_{I_{x^*},\tau})$ is the interior of $C_{I_{x^*},\tau}$. Let $U_{\alpha}$ be a ball of radius $\alpha >0$, centered at $x^* - \nabla f(x^*)/L$, such that $U_{\alpha}\subseteq \operatorname{int}(C_{I_{x^*},\tau})$. By continuity, we may choose an 
$\epsilon >0$ such that $\|x - x^*\|_2<\epsilon$ implies $x - \nabla f(x)/L\in U_{\alpha}$.

Now choose $\beta >0$ such that for all $\|x\|_1\leq\tau$, $f(x) - f(x^*) < \beta$ implies $\|x-x^*\|_2<\epsilon$. To see that $\beta > 0$ exists, suppose for a contradiction that $\forall \ n$, $\exists \ x_n$ with $\|x_n\|_1\leq\tau$ where $f(x_n) - f(x^*) <1/n$ but $\|x_n-x^*\|_2 \geq\epsilon$. Since the $1$-norm ball is compact, there is a subsequence $\{x_{n_k}\}$ of $\{x_n\}$ converging to some $x'$. By continuity, $f(x_{n_k})$ converges to $f(x')$. As mentioned in the proof of Lemma~\ref{yconverge}, $x^*$ is a unique minimum. Thus, $f(x')\neq f(x^*)$ contradicting the assumption that $f(x_n)$ converges to $f(x^*)$.

We now show that Algorithm~\ref{alg:nesta-update} is linearly convergent if the initial prox-center $p_0$ is close enough to $x^*$. Suppose $\|p_0 - x^*\|_2 < \beta/L$. Then $(\ref{eq:Lip})$ implies 
\begin{equation*}
f(y_{1K}) - f(x^*) \leq \frac{L}{K(K+1)}\|p_0 - x^*\|_2^2 < \beta, 
\end{equation*} 
and so $\|y_{1K}-x^*\|<\epsilon$. By Lemma~\ref{sparse1}, $p_1 = \operatorname{proj}_{1}(y_{1K}-\nabla f(y_{1K})/L,\tau)$ is $s$-sparse, and by $(\ref{eqn1})$, 
\begin{equation}\label{e1}
\delta\|p_1 - x^*\|_2^2\leq f(p_1) -f(x^*).
\end{equation}

Note that $p_1$ is the result of a step of the \textit{projected gradient method}, i.e. $x_{k+1} = \operatorname{proj}_{1}(x_{k}-\nabla f(x_{k})/L,\tau)$. Since this method is monotonically decreasing (see \cite{PGMmonotonic} for a proof),  
\begin{equation}\label{e2}
f(p_1) -f(x^*) \leq f(y_{1K}) -f(x^*).
\end{equation}

Combining $(\ref{e1})$ and $(\ref{e2})$ with $(\ref{eq:Lip})$, gives 
\begin{equation*}
\|p_1 - x^*\|_2 \leq \frac{1}{K}\sqrt{\frac{L}{\delta}}\|p_0 - x^*\|_2.
\end{equation*}
Since we assume that $K> \sqrt{\frac{L}{\delta}}$, we have $\|p_1 - x^*\|_2 < \beta/L$. Thus, the above arguments can be repeatedly applied to show that for any $j$,
\begin{equation}\label{e3}
\|p_j - x^*\|_2 \leq \left(\frac{1}{K}\sqrt{\frac{L}{\delta}}\right)^j\|p_0 - x^*\|_2.
\end{equation}

\item First observe that \eqref{e3} implies
\[
\lVert p_{j}-x^{\ast}\rVert_{2}\leq\left(  \frac{1}{K}\sqrt{\frac{L}{\delta}%
}\right)  ^{j}\lVert p_{0}-x^{\ast}\rVert_{2}\leq\varepsilon\lVert
p_{0}-x^{\ast}\rVert_{2}%
\]
when%
\[
j\log\left(  \frac{1}{K}\sqrt{\frac{L}{\delta}}\right)  =\log\varepsilon.
\]
This relation allows us to choose $K$ to minimize the product $j\cdot K$. Since%
\[
j\cdot K=\frac{K\log\varepsilon}{\log\sqrt{L/\delta}-\log K},
\]
taking derivative of the expression on the right shows that $j\cdot K$ is minimized
when%
\[
K =e\sqrt{\frac{L}{\delta}},%
\]
where $e$ is the base of the natural logarithm. The total number of iterations will then be
\[
j_{\operatorname*{tot}}\cdot K=-e\sqrt{\frac{L}{\delta}}%
\log\varepsilon.
\]\qed
\end{enumerate}\end{proof}

\noindent Theorem~\ref{remark:direction} implies that we almost always have local linear convergence:
\begin{corollary}\label{cor:OPT}
If $A$ satisfies the \textit{restricted isometry property} of order $2s$ and the
solution $x^{\ast}$ is $s$-sparse, Algorithm~\ref{alg:nesta-update} is almost always locally linearly convergent for any $K>\sqrt{\frac{L}{\delta}}$.
\end{corollary}

In our experiments, there are some cases where updating the prox-center will
eventually cause the duality gap to jump to a higher value than the previous
iteration. This can cause the algorithm to run for more iterations than
necessary. A check is added to prevent the prox-center from being updated if
it no longer helps.

In Table~\ref{table1}, we give some results showing that updating the prox-center is effective when using \textsc{nesta-lasso} to solve the \textsc{lasso} problem.

\begin{table}
\centering
\caption{Number of products with $A$ and $A^{\top}$ for \textsc{nesta-lasso} without prox-center updates (cf. Algorithm~\ref{alg:nesta}) and \textsc{nesta-lasso} with prox-center updates (cf. Algorithm~\ref{alg:nesta-lasso}). These values are given by $N_A$ and $N_{A}^{\text{update}}$ respectively.}
\label{table1}
\begin{tabular}{lllll}
\hline\noalign{\smallskip}
Number of Rows of $A$ & Number of Columns of $A$ & $\tau$ \ & $N_{A}$ \ & $ N_{A}^{\text{update}}$ \\
\noalign{\smallskip}\hline\noalign{\smallskip}
100 & 256 & 6.28 & 69 & 37\\
200 & 512 & 12.6 & 77 & 47\\
400 & 1024 & 25.1 & 157 & 45\\
\noalign{\smallskip}\hline
\end{tabular}
\end{table}


\section{PARNES\label{sec:PARNES}} 

In applications where the noise level of the problem is approximately known, it is preferable to solve \textsc{bp}$(\sigma)$. The Pareto root-finding method used by van~den~Berg and Friedlander \cite{SPGL1} interprets \textsc{bp}$(\sigma)$ as finding the root of a single-variable nonlinear equation whose graph is called the Pareto curve. Their implementation of this approach is called \textsc{spgl1}. In \textsc{spgl1}, an inexact version of Newton's method is used to find the root, and at each iteration, an approximate solution to the \textsc{lasso} problem, \textsc{ls}$(\tau)$, is found using an \textsc{spg} approach. Refer to \cite{inexactNewt} for more information on the inexact Newton method. In Section~\ref{sec:NUM}, we show experimentally that using \textsc{nesta}-\textsc{lasso} in place of the \textsc{spg} approach for solving the \textsc{ls}$(\tau)$ subproblems can lead to improved results. We call this version of the Pareto root-finding method, \textsc{parnes}. The pseudocode of \textsc{parnes} is given in Algorithm~\ref{alg:parnes}.

\subsection{Pareto curve}
\label{sec:ParCurve}

Suppose $A$ and $b$ are given, with $0\neq b\in\text{range}(A)$. The points on the Pareto curve are given by $(\tau,\varphi(\tau))$ where $\varphi(\tau)=\left\Vert Ax_{\tau}-b\right\Vert
_{2}$, $\tau = \left\Vert x_{\tau}\right\Vert _{1}$, and $x_{\tau}$ solves \textsc{ls}$(\tau)$. The Pareto
curve gives the optimal trade-off between the 2-norm of the residual and
1-norm of the solution to \textsc{ls}$(\tau)$. It can also be shown that the Pareto
curve also characterizes the optimal trade-off between the 2-norm of the
residual and 1-norm of the solution to \textsc{bp}$(\sigma)$. Refer to \cite{SPGL1} for a more detailed explanation of these properties of the Pareto curve. An example of a Pareto curve is shown in Figure~\ref{fig:1}.

Let $\tau_{\textsc{bp}}$ be the optimal objective value of \textsc{bp}$(0)$. The Pareto curve is restricted to the interval $\tau\in[0,\tau_{\textsc{bp}}]$ with $\varphi(0)=\left\Vert b\right\Vert _{2}>0$ and
$\varphi(\tau_{\textsc{bp}})=0$. The following theorem, proven by van den Berg and Friedlander, shows that the Pareto curve is convex, strictly decreasing over
the interval $\tau\in[0,\tau_{\textsc{bp}}]$, and continuously differentiable
for $\tau\in(0,\tau_{\textsc{bp}})$.

\begin{proposition}\textbf{\cite{SPGL1}}
\label{paretoProp} The function $\varphi$ is
\begin{enumerate}[\upshape (i)]
\item convex and nonincreasing;

\item continuously differentiable for $\tau\in(0,\tau_{\textsc{bp}})$ with
$\varphi^{\prime}(\tau)=-\lambda_{\tau}$ where $\lambda_{\tau}=\lVert
A^{T}y_{\tau}\rVert _{\infty}$ is the optimal dual variable to
\textsc{ls}$(\tau)$ and $y_{\tau}=r_{\tau}/\lVert r_{\tau}\rVert
_{2}$ with $r_{\tau} = Ax_{\tau}-b$;

\item strictly decreasing and $\lVert x_{\tau}\rVert _{1} =\tau$ for
$\tau\in[0,\tau_{\textsc{bp}}]$.
\end{enumerate}
\end{proposition}

\begin{figure*}
\centering
\includegraphics[scale=0.5]{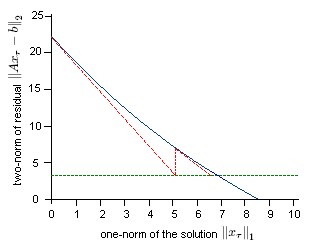}
\caption{An example of a Pareto curve. The solid line is the Pareto curve; the dotted red lines give two
iterations of Newton's method.}
\label{fig:1}
\end{figure*}

\subsection{Root finding}
\label{sec:RootFind}
Since the Pareto curve characterizes the optimal trade-off for both
\textsc{bp}$(\sigma)$ and \textsc{ls}$(\tau)$, solving \textsc{bp}$(\sigma)$
for a fixed $\sigma$ can be interpreted as finding a root of the non-linear
equation $\varphi(\tau)=\sigma$. The iterations consist of finding the
solution to \textsc{ls}$(\tau)$ for a sequence of parameters $\tau
_{k}\rightarrow\tau_{\sigma}$ where $\tau_{\sigma}$ is the optimal objective
value of \textsc{bp}$(\sigma)$.

Applying Newton's method to $\varphi$ gives
\[
\tau_{k+1}=\tau_{k}+(\sigma-\varphi(\tau_{k}))/\varphi^{\prime}(\tau_{k}).
\]
Since $\varphi$ is convex, strictly decreasing and continuously
differentiable, $\tau_{k}\rightarrow\tau_{\sigma}$ superlinearly for all
initial values $\tau_{0}\in(0,\tau_{\textsc{bp}})$ (see Proposition~1.4.1 in
\cite{bertsekas}). By Proposition~\ref{paretoProp}, $\varphi(\tau_{k})$ is the
optimal value to \textsc{ls}$(\tau_{k})$ and $\varphi^{\prime}(\tau_{k})$ is
the dual solution to \textsc{ls}$(\tau_{k})$. Since evaluating $\varphi(\tau_{k})$ involves solving a potentially large
optimization problem, an inexact Newton method
is carried out with approximations of $\varphi(\tau_{k})$ and $\varphi
^{\prime}(\tau_{k})$.

Let $\overline{y}_{\tau}$ and $\overline{\lambda}_{\tau}$ be the approximations of the
$y_{\tau}$ and $\lambda_{\tau}$ defined in Proposition~\ref{paretoProp}. The duality gap at each iteration is given by
\[
\eta_{\tau}=\left\Vert \overline{r}_{\tau}\right\Vert _{2}-(b^{T}%
\overline{y}_{\tau}-\tau\overline{\lambda}_{\tau}).
\]
The following convergence result has been proven by van den Berg and Friedlander.

\begin{theorem}\textbf{\cite{SPGL1}}
Suppose $A$ has full rank, $\sigma\in(0,\left\Vert b\right\Vert _{2})$, and
the inexact Newton method generates a sequence $\tau_{k}\rightarrow
\tau_{\sigma}$. If $\eta_{k}:=\eta_{\tau_{k}}\rightarrow0$ and $\tau_{0}$
is close enough to $\tau_{\sigma}$, we have
\[
|\tau_{k+1}-\tau_{\sigma}|=\gamma_{1}\eta_{k}+\zeta_{k}|\tau_{k}-\tau
_{\sigma}|,
\]
where $\zeta_{k}\rightarrow0$ and $\gamma_{1}$ is a positive constant.
\end{theorem}

\subsection{Solving the LASSO problem}
\label{sec:SolveLASSO}
Approximating $\varphi(\tau_{k})$ and $\varphi^{\prime}(\tau_{k})$ require
approximately minimizing \textsc{ls}$(\tau)$. The solver \textsc{spgl1}
uses a spectral projected-gradient (\textsc{spg}) algorithm. T	he method follows the
algorithm by Birgin, Mart\'{\i}nez, and Raydan \cite{Birgin} and is shown to
be globally convergent. The costs include evaluating $Ax$, $A^{\top}r$, and a
projection onto the $1$-norm ball $\left\Vert x\right\Vert _{1}\leq\tau$. In
\textsc{parnes}, we replace this \textsc{spg} algorithm with our algorithm, \textsc{nesta}-\textsc{lasso}
(cf.\ Algorithm~\ref{alg:nesta-lasso}).

\begin{algorithm}[htbp]
\caption{\textsc{parnes}: Pareto curve method with \textsc{nesta}-\textsc{lasso}}
\label{alg:parnes}
\begin{algorithmic}[1]
\REQUIRE{initial point $x_0$, \textsc{bpdn} parameter $\sigma$, tolerance $\eta$.}
\ENSURE{$x_{\sigma} = \operatorname*{argmin} \{\lVert x \rVert_1 : \lVert Ax - b \rVert_2 \le \sigma \}$}
\STATE $\tau_{0}=0$, $\varphi_{0} = \lVert b \rVert_2$, $\varphi_{0}' = \lVert A^\top b \rVert_\infty$;
\FOR{$k=0,\dots,k_{\max}$,}
\STATE $\tau_{k+1} = \tau_{k}+(\sigma-\varphi_{k})/\varphi_{k}'$;
\STATE $x_{k+1} = $ \textsc{nesta}-\textsc{lasso}$(x_{k},\tau_{k+1},\eta)$;
\STATE $r_{k+1} = b - Ax_{k+1}$;
\STATE $\varphi_{k+1} = \lVert r_{k+1}\rVert_{2}$;
\STATE $\varphi_{k+1}' = -\lVert A^\top r_{k+1}\rVert_{\infty}/\lVert r_{k+1}\rVert_{2}$;
\IF{$\lVert r_{k+1}\rVert_{2} -\sigma \leq\eta \cdot \max\{1,\lVert r_{k+1}\rVert_{2}\}$}
\RETURN $x_{\sigma} = x_{k+1}$;
\ENDIF
\ENDFOR
\end{algorithmic}
\end{algorithm}

\section{Other solution techniques and tools\label{sec:SOLVERS}}

In the our numerical experiments, we compare \textsc{parnes} with other state-of-the-art methods. The algorithms we test and their experimental details are described below. Note that the algorithms either solve \textsc{bp}$(\sigma)$ or \textsc{qp}$(\lambda)$.

\subsection{NESTA \cite{NESTA}}
\label{sec:NESTA_INFO}
\textsc{NESTA} is used to solve \textsc{bp}$(\sigma)$. Its code is available at
\url{http://www.acm.caltech.edu/~nesta}. The parameters for \textsc{nesta} are
set to be
\[
x_{0} = A^{\top}b,\quad\mu= 0.02,
\]
where $x_{0}$ is the initial guess and $\mu$ is the smoothing parameter for
the $1$-norm function in \textsc{bp}$(\sigma)$.

Continuation techniques are used to speed up \textsc{nesta} in
\cite{NESTA}. Such techniques are useful when it is observed that a problem
involving some parameter $\lambda$ is faster for large $\lambda$, \cite{homotopy1,fpc1}. Thus, the idea of continuation is to solve a sequence of problems
for decreasing values of $\lambda$. In the case of \textsc{nesta}, it is
observed that convergence is faster for larger values of $\mu$. When
continuation is used in the experiments, there are four continuation steps
with $\mu_{0} = \|x_{0}\|_{\infty}$ and $\mu_{t} = (\mu/\mu_{0})^{t/4}\mu_{0}$
for $t = 1,2,3,4$.

\subsection{GPSR: Gradient projection for sparse reconstruction \cite{GPSR}}
\label{sec:GPSR}
\textsc{gpsr} is used to solve the penalized least-squares problem
\textsc{qp}$(\lambda)$. The code is available at
\url{http://www.lx.it.pt/~mtf/GPSR}. The problem is first recast as a
bound-constrained quadratic program (\textsc{bcqp}) by using a standard change
of variables on $x$. Here, $x = u_{1}-u_{2}$, and the variables are now given
by $[u_{1},u_{2}]$ where the entries are positive. The new problem is then
solved using a gradient projection (\textsc{gp}) algorithm. The parameters are set to
the default values in the following experiments.

A version of \textsc{gpsr} with continuation is also tested. The number of
continuation steps is 40, the variable \textsc{tolerancea} is set to $10^{-3}$,
and the variable \textsc{minitera} is set to 1. All other parameters are set
to their default values. 

\subsection{SpaRSA: Sparse reconstruction by separable approximation
\cite{SpaRSA}}
\label{sec:SpaRSA}
\textsc{sparsa} is used to minimize functions of the form $\phi(x) = f(x) +
\lambda c(x)$ where $f$ is smooth and $c$ is non-smooth and non-convex. The
\textsc{qp}$(\lambda)$ problem is a special case of functions of this form.
The code for \textsc{sparsa} is available at \url{http://www.lx.it.pt/~mtf/SpaRSA}.

In a sense, \textsc{sparsa} is an iterative shrinkage/thresholding algorithm. Utilizing continuation and a Brazilai-Borwein heuristic \cite{bb} to find step sizes, the speed of the algorithm can be
increased. The number of continuation steps is set to 40 and the variable
\textsc{minitera} is set to 1. All remaining variables are set to their
default values.

\subsection{SPGL1 \cite{SPGL1} and SPARCO \cite{SPARCO}}
\label{sec:SPGL1}
\textsc{SPGL1} is available at \url{http://www.cs.ubc.ca/labs/scl/spgl1}. The parameters for our numerical
experiments are set to their default values.

Due to the vast number of available and upcoming algorithms for sparse
reconstruction, the authors of \textsc{spgl1} and others have created
\textsc{sparco} \cite{SPARCO}. In \textsc{sparco}, they provide a much needed testing
framework for benchmarking algorithms. It consists of a large collection of
imaging, compressed sensing, and geophysics problems. Moreover, it includes a
library of standard operators which can be used to create new test problems.
\textsc{sparco} is implemented in \textsc{matlab} and was originally created to test \textsc{spgl1}. The toolbox is available at \url{http://www.cs.ubc.ca/labs/scl/sparco}.

\subsection{FISTA: Fast iterative soft-thresholding algorithm \cite{FISTA}}
\label{sec:FISTA}
\textsc{fista} solves \textsc{qp}$(\lambda)$. It can be thought of as a
simplified version of the Nesterov algorithm in Section~\ref{sec:NESTEROV} since it involves two sequences of iterates
instead of three. In Section~4.2 of \cite{NESTA}, \textsc{fista} is shown to give very accurate solutions provided enough iterations are taken. Due to its ease of use and accuracy, \textsc{fista} is used to compute reference solutions in \cite{NESTA} and in this paper. The code for \textsc{fista} can be found in the \textsc{nesta} experiments code at \url{http://www.acm.caltech.edu/~nesta}.

\subsection{FPC: Fixed point continuation \cite{fpc1,fpc2}}
\label{sec:FPC}
\textsc{fpc} solves the general problem $\min_{x} \left\|  x\right\| _{1}
+\lambda f(x)$ where $f(x)$ is differentiable and convex. The
special case with $f(x) = \frac{1}{2}\left\|  Ax-b\right\| _{2}^{2}$ is the
\textsc{qp}$(\lambda)$ problem. The algorithm is available at \url{http://www.caam.rice.edu/~optimization/L1/fpc}.

\textsc{FPC} is equivalent to iterative soft-thresholding. The approach is based on the observation that the solution solves a fixed-point equation $x = F(x)$ where the operator $F$ is a composition of a
gradient descent-like operator and a shrinkage operator. It can be shown that
the algorithm has $q$-linear convergence and also, finite-convergence for some
components of the solution. Since the parameter $\lambda$ affects the speed of
convergence, continuation techniques are used to slowly decrease $\lambda$ for
faster convergence. A more recent version of \textsc{fpc}, \textsc{fpc-bb},
uses Brazilai-Borwein steps to speed up convergence. Both versions of
\textsc{fpc} are tested with their default parameters.

\subsection{FPC-AS: Fixed-point continuation and active set \cite{FPC-AS}}
\label{sec:FPC-AS}
\textsc{fpc-as} is an extension of \textsc{fpc} into a two-stage algorithm
which solves \textsc{qp}$(\lambda)$. The code can be found at
\url{http://www.caam.rice.edu/~optimization/L1/fpc}. It has been shown in
\cite{fpc1} that applying the shrinkage operator a finite number of times
yields the support and signs of the optimal solution. Thus, the first stage of
\textsc{fpc-as} involves applying the shrinkage operator until an active set
is determined. In the second stage, the objective function is restricted to
the active set and $\left\| x\right\| _{1}$ is replaced by $c^{T}x$ where $c$
is the vector of signs of the active set. The constraint $c_{i}\cdot x_{i} >0$
is also added. Since the objective function is now smooth, many available
methods can now be used to solve the problem. In the following tests, the
solvers \textsc{l-bfgs} and conjugate gradients, \textsc{cg} (referred to as
\textsc{fpc-as} (\textsc{cg})), are used. Continuation methods are used to
decrease $\lambda$ to increase speed. For experiments involving approximately
sparse signals, the parameter controlling the estimated number of nonzeros is
set to $n$, and the maximum number of subspace iterations is set to 10. The
other parameters are set to their default values. All other experiments were
tested with the default parameters.

\subsection{Bregman iteration \cite{Bregman}}
\label{sec:Bregman}
The Bregman Iterative algorithm consists of solving a sequence of
\textsc{qp}$(\lambda)$ problems for a fixed $\lambda$ and updated observation
vectors $b$. Each \textsc{qp}$(\lambda)$ is solved using the Brazilai-Borwein
version of \textsc{fpc}. Typically, very few (around four) outer iterations
are needed. Code for the Bregman algorithm can be found at
\url{http://www.caam.rice.edu/~optimization/L1/2006/10/bregman-iter}\\
\url{ative-algorithms-for.html}. All parameters are set to their default values.

\subsection{C-SALSA \cite{SALSA1,SALSA}}
\label{sec:SALSA}
This state-of-the-art method solves \textsc{bp}$(\sigma)$ and has been shown to be competitive with \textsc{spgl1} and \textsc{nesta}. The method solves the general constrained optimization problem 
\[\min_x \phi(x) \text{ s.t. }\|Ax-b\|_2\leq\epsilon.\] 
First, the method transforms the problem into an unconstrained problem which is then transformed into a different constrained problem and then solved with an augmented Lagrangian scheme. 

The algorithm requires a method to compute the inverse of $(A^{\top}A + \alpha I)$ with $\alpha > 0$ and an efficient method for computing the denoising operator associated with $\phi$. We have hand-tuned the parameters $\mu_1$ and $\mu_2$ for optimal performance. The code for $\textsc{c-salsa}$ can be found at \url{http://cascais.lx.it.pt/~mafonso/salsa.html}.

\section{Numerical results\label{sec:NUM}}
In the \textsc{nesta} paper \cite{NESTA} extensive experiments are carried out, comparing the effectiveness of the state-of-the-art sparse reconstruction algorithms described in Section~\ref{sec:SOLVERS}. The code used to run these experiments is available at \url{http://www.acm.caltech.edu/~nesta}. We have modified this \textsc{nesta} experiment infrastructure to include \textsc{parnes} and \textsc{c-salsa}, and we repeat some of the tests in \cite{NESTA} using the same experimental standards and parameters. Refer to the \cite{NESTA} for a detailed description of the experiments. 

One difficulty that arises in carrying out such broad experiments is that some of the algorithms solve \textsc{qp}$(\lambda)$ whereas others solve \textsc{bp}$(\sigma)$. Comparing the algorithms
thus requires a way of finding a $(\sigma,\lambda)$ pair for which the
solutions of \textsc{qp}$(\lambda)$ and \textsc{bp}$(\sigma)$ coincide. The \textsc{nesta} experiments utilize a two-step procedure. Given the noise level $\epsilon$, the authors choose $\sigma_0 := \sqrt{m+2\sqrt{2m}}\epsilon$, and then use \textsc{spgl1} to solve the corresponding $\textsc{bp}(\sigma_0)$ problem. The \textsc{spgl1} dual solution then provides an estimate of the corresponding $\lambda$. In practice, the computation of $\lambda$ is not very stable, and so a second step is performed in which \textsc{fista} is used to compute a $\sigma$ corresponding to $\lambda$ using a very high accuracy of around $10^{-14}$. 

The highly accurate solution computed by \textsc{fista} is used to determine the accuracy of the solutions computed by the other solvers. Section~4.2 of \cite{NESTA} shows that this is reasonable since \textsc{fista} gives very accurate solutions provided that enough iterations are taken. For each test, \textsc{fista} is ran twice. In the first run, \textsc{fista} is ran with no limit on the number of iterations until the relative change in the function value is less than $10^{-14}$. This solution is used to determine the accuracy of the computed solutions. The results recorded for \textsc{fista} are from running \textsc{fista} a second time with either stopping criterion \eqref{stopcond} or \eqref{stopcond3}.

Since the different algorithms utilize different stopping criteria, to maintain fairness, the codes have been modified to allow for two new stopping criteria. Intuitively, the algorithms are run until they achieve a solution at least as accurate as the one obtained by \textsc{nesta}. In \cite{NESTA}, \textsc{nesta} (with continuation) is used to compute a solution $x_{\text{NES}}$. Let $\hat{x}_{k}$ be the $k$-th iteration in the algorithm being tested. The stopping criteria used are:

\begin{equation}
\Vert\hat{x}_{k}\Vert_{\ell_{1}}\leq\Vert x_{\text{NES}}\Vert_{\ell_{1}}%
\quad\text{and}\quad\Vert b-A\hat{x}_{k}\Vert_{\ell_{2}}\leq1.05\,\Vert
b-Ax_{\text{NES}}\Vert_{\ell_{2}}, \label{stopcond}%
\end{equation}
and
\begin{equation}
\lambda\Vert\hat{x}_{k}\Vert_{\ell_{1}}+\frac{1}{2}\Vert A\hat{x}_{k}%
-b\Vert_{\ell_{2}}^{2}\leq\lambda\Vert x_{\text{NES}}\Vert_{\ell_{1}}+\frac
{1}{2}\Vert Ax_{\text{NES}}-b\Vert_{\ell_{2}}^{2}. \label{stopcond3}%
\end{equation}

The rationale for having two stopping criteria is to reduce any potential bias arising from the fact that some algorithms solve \textsc{qp}$(\lambda)$, for which \eqref{stopcond3} is the most natural, while
others solve \textsc{bp}$(\sigma)$, for which \eqref{stopcond} is the most
natural. It is evident from the tables below that there is not a significant
difference between using \eqref{stopcond} and \eqref{stopcond3}. For each test, the number of calls to $A$ and $A^{\top}$ $(N_A)$ is recorded, and the algorithms are said to have not converged (\textsc{dnc}) if the number of calls exceeds 20,000.

In Tables~\ref{table51} and \ref{table52}, we repeat the experiments done in
Tables~5.1 and 5.2 of \cite{NESTA}. These experiments
involve recovering an unknown, exactly $s$-sparse signal with $n = $
262,144, $m = n/8$, and $s = m/5$. For each run, the measurement operator $A$ is a randomly subsampled discrete
cosine transform, and the noise level is set to $0.1$. The experiments are performed with
increasing values of the dynamic range $d$ where $d = 20, 40, 60, 80, 100$ dB.

The dynamic range $d$ is a measure of the ratio between the
largest and smallest magnitudes of the non-zero coefficients of the unknown
signal. Problems with a high dynamic range occur often in applications. In
these cases, high accuracy becomes important since one must be able to detect
and recover low-power signals with small amplitudes which may be obscured by
high-power signals with large amplitudes. 

Table~\ref{table41} compares the accuracy of the different solvers when used to calculate the results in the last column of Table~\ref{table51}. As this corresponds to a very high dynamic range (100 dB), one hopes to obtain very accurate results. Although \textsc{fista} produces the most accurate results ($\|x-x^*\|_1/\|x^*\|_1 = 3.63\cdot 10^{-4}$), with at least twice the accuracy of the other solvers, it requires the over 10,000 calls to $A$ and $A^{\top}$. In contrast, \textsc{parnes} only requires 632 function calls to reach a relative accuracy of $\|x-x^*\|_1/\|x^*\|_1 = 6.93\cdot 10^{-4}$. The solvers \textsc{fpc-as} and \textsc{fpc-as (cg)} do well and only require around 300 iterations to reach a relative accuracy of around $6.93\cdot 10^{-4}$. The remaining algorithms reach relative accuracies of around $8\cdot 10^{-4}$ or more, and \textsc{gspr} does not converge. Without continuation, \textsc{nesta} only achieves a relative accuracy of $4.12\cdot 10^{-3}$ after 15,227 function calls. However, \textsc{nesta} with continuation does much better and reaches a relative accuracy of $8.12\cdot 10^{-4}$ after 787 function calls.

In Tables~\ref{table51} and \ref{table52}, the same experiment is ran for the two stopping criteria. Since there is not a notable difference between the two sets of results, we only analyze Table~\ref{table51}. Here, \textsc{fpc-as} and \textsc{fpc-as (cg)} perform the best for large values of $d$, and the number of function calls mostly range from 200 to 375 for all values of the dynamic range. In these cases, we see a relatively small increase in $N_A$ as $d$ increases from 20 dB to 100 dB. Our method, \textsc{parnes}, and \textsc{spgl1} generally perform well and do particularly well for small values of $d$. However, both exhibit a larger increase in $N_A$ with $d$, with \textsc{parnes} increasing from 122 to 632 function calls and \textsc{spgl1} ranging between 58 and 504. The solvers \textsc{nesta + ct} and \textsc{sparsa} perform relatively well for large values of $d$ with $N_A$ ranging between 500 and 800.

In applications, the signal to be recovered is often approximately sparse rather than exactly sparse. Again, high accuracy is
important when solving these problems. The last two tables, Tables~\ref{table53} and \ref{table53-1}, replicate
Tables 5.3 and 5.4 of \cite{NESTA}. Each run involves an approximately sparse signal obtained from a permutation
of the Haar wavelet coefficients of a $512\times512$ image. The measurement
vector $b$ consists of $m = n/8 = 512^{2}/8 =$ 32,768 random discrete cosine
measurements, and the noise level is set to have a variance of $1$ in Table~\ref{table53} and $0.1$ in Table~\ref{table53-1}. For more specific details, refer to \cite{NESTA}. 

We have seen that \textsc{nesta + ct}, \textsc{sparsa}, \textsc{spgl1},
\textsc{parnes}, and both versions of \textsc{fpc-as} perform well in the case
of exactly sparse signals for all values of the dynamic range. However, in the
case of approximately sparse signals, \textsc{sparsa} and all versions of \textsc{fpc} no longer converge in under 20,000 function calls. In Table~\ref{table53}, \textsc{parnes}, \textsc{spgl1}, and \textsc{c-salsa} perform well, with \textsc{parnes} and \textsc{c-salsa} taking around 650 function calls for some runs (compare to \textsc{nesta + ct} which takes at least 3,000 iterations). These algorithms also perform the best in Table~\ref{table53-1}, and most other algorithms no longer converge in under 10,000 function calls.  

\begin{table}
\centering
\caption{Comparison of accuracy using experiments from Table \ref{table51}.
Dynamic range $100$ dB, $\sigma= 0.100$, $\mu= 0.020$, sparsity level $s =
m/5$. Stopping rule is \eqref{stopcond}.}%
\label{table41}%
\begin{tabular}{lllllll}%
\hline\noalign{\smallskip}
Methods & $N_{A}$ & $\lVert x \rVert_{1}$ & $\lVert Ax - b\rVert_{2}$ &
$\frac{\lVert x - x^{*} \rVert_{1}}{\lVert x^{*} \rVert_{1}}$ & $\lVert x -
x^{*}\rVert_{\infty}$ & $\lVert x - x^{*}\rVert_{2}$\\
\noalign{\smallskip}\hline\noalign{\smallskip}
\textsc{parnes} & $632 $ & $942197.606$ & $2.692$ & $0.000693$ & $8.312 $ &
$46.623 $\\
\textsc{nesta} & $15227$ & $942402.960$ & $2.661$ & $0.004124$ & $45.753 $ &
$255.778 $\\
\textsc{nesta + ct} & $787 $ & $942211.581$ & $2.661$ & $0.000812$ & $9.317 $
& $52.729 $\\
\textsc{gpsr} & \textsc{dnc} & \textsc{dnc} & \textsc{dnc} & \textsc{dnc} &
\textsc{dnc} & \textsc{dnc}\\
\textsc{gpsr + ct} & $11737$ & $942211.377$ & $2.725$ & $0.001420$ & $15.646 $
& $90.493 $\\
\textsc{sparsa} & $693 $ & $942197.785$ & $2.728$ & $0.000783$ & $9.094 $ &
$51.839 $\\
\textsc{spgl1} & $504 $ & $942211.520$ & $2.628$ & $0.001326$ & $14.806 $ &
$84.560 $\\
\textsc{fista} & $12462$ & $942211.540$ & $2.654$ & $0.000363$ & $4.358 $ &
$26.014 $\\
\textsc{fpc-as} & $287 $ & $942210.925$ & $2.498$ & $0.000672$ & $9.374 $ &
$45.071 $\\
\textsc{fpc-as} (\textsc{cg}) & $361 $ & $942210.512$ & $2.508$ & $0.000671$ &
$9.361 $ & $45.010 $\\
\textsc{fpc} & $9614 $ & $942211.540$ & $2.719$ & $0.001422$ & $15.752 $ &
$90.665 $\\
\textsc{fpc-bb} & $1082 $ & $942209.854$ & $2.726$ & $0.001378$ & $15.271 $ &
$87.963 $\\
\textsc{bregman-bb} & $1408 $ & $942286.656$ & $1.326$ & $0.000891$ & $9.303 $
& $52.449 $ \\
\textsc{c-salsa} & $1338 $ & $942219.455$ & $2.317$ & $0.000851$ & $9.541 $
& $55.14 $ \\
\noalign{\smallskip}\hline
\end{tabular}
\end{table}

\begin{table}
\centering
\caption{Number of function calls where the sparsity level is $s = m/5$ and
the stopping rule is \eqref{stopcond}.}%
\label{table51}
\begin{tabular}{llllll}%
\hline\noalign{\smallskip}
Method & 20 dB & 40 dB & 60 dB & 80 dB & 100 dB\\
\noalign{\smallskip}\hline\noalign{\smallskip}
\textsc{parnes} & 122 & 172 & 214 & 470 & 632\\
\textsc{nesta} & 383 & 809 & 1639 & 4341 & 15227\\
\textsc{nesta + ct} & 483 & 513 & 583 & 685 & 787\\
\textsc{gpsr} & 64 & 622 & 5030 & \textsc{dnc} & \textsc{dnc}\\
\textsc{gpsr + ct} & 271 & 219 & 357 & 1219 & 11737\\
\textsc{sparsa} & 323 & 387 & 465 & 541 & 693\\
\textsc{spgl1} & 58 & 102 & 191 & 374 & 504\\
\textsc{fista} & 69 & 267 & 1020 & 3465 & 12462\\
\textsc{fpc-as} & 209 & 231 & 299 & 371 & 287\\
\textsc{fpc-as} (\textsc{cg}) & 253 & 289 & 375 & 481 & 361\\
\textsc{fpc} & 474 & 386 & 478 & 1068 & 9614\\
\textsc{fpc-bb} & 164 & 168 & 206 & 278 & 1082\\
\textsc{bregman-bb} & 211 & 223 & 309 & 455 & 1408\\
\textsc{c-salsa} & 242 & 602 & 702 & 970 & 1338\\
\noalign{\smallskip}\hline
\end{tabular}
\end{table}

\begin{table}
\centering
\caption{Number of function calls where the sparsity level is $s = m/5$ and
the stopping rule is \eqref{stopcond3}.}%
\label{table52}
\begin{tabular}{llllll}%
\hline\noalign{\smallskip}
Method & 20 dB & 40 dB & 60 dB & 80 dB & 100 dB\\
\noalign{\smallskip}\hline\noalign{\smallskip}
\textsc{parnes} & 74 & 116 & 166 & 364 & 562\\
\textsc{nesta} & 383 & 809 & 1639 & 4341 & 15227\\
\textsc{nesta + ct} & 483 & 513 & 583 & 685 & 787\\
\textsc{gpsr} & 62 & 618 & 5026 & \textsc{dnc} & \textsc{dnc}\\
\textsc{gpsr + ct} & 271 & 219 & 369 & 1237 & 11775\\
\textsc{sparsa} & 323 & 387 & 463 & 541 & 689\\
\textsc{spgl1} & 43 & 99 & 185 & 365 & 488\\
\textsc{fista} & 72 & 261 & 1002 & 3477 & 12462\\
\textsc{fpc-as} & 115 & 167 & 159 & 371 & 281\\
\textsc{fpc-as} (\textsc{cg}) & 142 & 210 & 198 & 481 & 355\\
\textsc{fpc} & 472 & 386 & 466 & 1144 & 9734\\
\textsc{fpc-bb} & 164 & 164 & 202 & 276 & 1092\\
\textsc{bregman-bb} & 211 & 223 & 309 & 455 & 1408\\
\textsc{c-salsa} & 202 & 550 & 650 & 898 & 1230\\
\noalign{\smallskip}\hline
\end{tabular}
\end{table}

\begin{table}
\centering
\caption{Recovery results of an approximately sparse signal (with Gaussian
noise of variance $1$ added) and with \eqref{stopcond3} as a stopping rule.}%
\label{table53}
\begin{tabular}{llllll}%
\hline\noalign{\smallskip}
Method & Run 1 & Run 2 & Run 3 & Run 4 & Run 5\\
\noalign{\smallskip}\hline\noalign{\smallskip}
\textsc{parnes} & 838 & 810 & 1038 & 1098 & 654\\
\textsc{nesta} & 8817 & 10867 & 9887 & 9093 & 11211\\
\textsc{nesta + ct} & 3807 & 3045 & 3047 & 3225 & 2735\\
\textsc{gpsr} & \textsc{dnc} & \textsc{dnc} & \textsc{dnc} & \textsc{dnc} &
\textsc{dnc}\\
\textsc{gpsr + ct} & \textsc{dnc} & \textsc{dnc} & \textsc{dnc} & \textsc{dnc}
& \textsc{dnc}\\
\textsc{sparsa} & 2143 & 2353 & 1977 & 1613 & \textsc{dnc}\\
\textsc{spgl1} & 916 & 892 & 1115 & 1437 & 938\\
\textsc{fista} & 3375 & 2940 & 2748 & 2538 & 3855\\
\textsc{fpc-as} & \textsc{dnc} & \textsc{dnc} & \textsc{dnc} & \textsc{dnc} &
\textsc{dnc}\\
\textsc{fpc-as} (\textsc{cg}) & \textsc{dnc} & \textsc{dnc} & \textsc{dnc} &
\textsc{dnc} & \textsc{dnc}\\
\textsc{fpc} & \textsc{dnc} & \textsc{dnc} & \textsc{dnc} & \textsc{dnc} &
\textsc{dnc}\\
\textsc{fpc-bb} & 5614 & 7906 & 5986 & 4652 & 6906\\
\textsc{bregman-bb} & 3288 & 1281 & 1507 & 2892 & 3104\\
\textsc{c-salsa} & 742 & 626 & 630 & 1226 & 826 \\
\noalign{\smallskip}\hline
\end{tabular}
\end{table}

\begin{table}
\centering
\caption{Recovery results of an approximately sparse signal (with Gaussian
noise of variance $0.1$ added) and with \eqref{stopcond3} as a stopping rule.}%
\label{table53-1}
\begin{tabular}{llllll}%
\hline\noalign{\smallskip}
Method & Run 1 & Run 2 & Run 3 & Run 4 & Run 5\\
\noalign{\smallskip}\hline\noalign{\smallskip}
\textsc{parnes} & 1420 & 1772 & 1246 & 1008 & 978\\
\textsc{nesta} & 11573 & 10457 & 10705 & 8807 & 13795\\
\textsc{nesta + ct} & 7543 & 13655 & 11515 & 3123 & 2777\\
\textsc{gpsr} & \textsc{dnc} & \textsc{dnc} & \textsc{dnc} & \textsc{dnc} &
\textsc{dnc}\\
\textsc{gpsr + ct} & \textsc{dnc} & \textsc{dnc} & \textsc{dnc} & \textsc{dnc}
& \textsc{dnc}\\
\textsc{sparsa} & 12509 & \textsc{dnc} & \textsc{dnc} & 3117 & \textsc{dnc}\\
\textsc{spgl1} & 1652 & 1955 & 2151 & 1311 & 2365\\
\textsc{fista} & 10845 & 12165 & 10050 & 7647 & 11997\\
\textsc{fpc-as} & \textsc{dnc} & \textsc{dnc} & \textsc{dnc} & \textsc{dnc} &
\textsc{dnc}\\
\textsc{fpc-as} (\textsc{cg}) & \textsc{dnc} & \textsc{dnc} & \textsc{dnc} &
\textsc{dnc} & \textsc{dnc}\\
\textsc{fpc} & \textsc{dnc} & \textsc{dnc} & \textsc{dnc} & \textsc{dnc} &
\textsc{dnc}\\
\textsc{fpc-bb} & \textsc{dnc} & \textsc{dnc} & \textsc{dnc} & \textsc{dnc} &
\textsc{dnc}\\
\textsc{bregman-bb} & 3900 & 3684 & 2045 & 3292 & 3486\\
\textsc{c-salsa} & 1886 & 1926 & 1770 & 1754 & 1854\\
\noalign{\smallskip}\hline
\end{tabular}
\end{table}

\subsection{Choice of parameters}
\label{sec:parameters}
As Tseng observed, accelerated proximal gradient algorithms will converge so
long as the condition given as equation (45) in \cite{Tseng} is satisfied. In
our case this translates into
\begin{equation}
\label{eq:Tseng}\min_{x \in\mathbb{R}^{n}}\left\{  \nabla f(y_{k})^{\top}x +
\frac{L}{2}\lVert x-x_{k}\rVert_{2}^{2} + P(x)\right\}  \geq\nabla
f(y_{k})^{\top}y_{k} + P(y_{k}),
\end{equation}
upon setting $\gamma_{k}=1$ and
\[
P(x) =
\begin{cases}
0 & \text{if } \lVert x \rVert_{1} \le\tau,\\
\infty & \text{otherwise},
\end{cases}
\]
in (45) in \cite{Tseng}. In other words, the value of $L$ need not necessarily
be fixed at the Lipschitz constant of $\nabla f$ but may be decreased, and 
decreasing $L$ has the same effect as increasing the stepsize. Tseng suggests
to decrease $L$ adaptively by a constant factor until (45) is violated,
then backtrack and repeat the iteration (cf.\ Note 6 in \cite{Tseng}). For
simplicity, and very likely at the expense of speed, we do not change our $L$
adaptively in \textsc{parnes} and \textsc{nesta}-\textsc{lasso}. Instead, we choose a small fixed $L$ by
trying a few different values so that \eqref{eq:Tseng} is satisfied for all
$k$ and likewise for the tolerance $\eta$ in Algorithm~\ref{alg:nesta-lasso}. However, even with this crude way of selecting $L$ and $\eta$, the results
obtained are still rather encouraging.

\section{Conclusions}
\label{sec:conclusion}
As seen in the numerical results, \textsc{spgl1} and \textsc{nesta} are among some of the top performing solvers available for basis pursuit denoising problems. We have therefore made use of Nesterov's accelerated proximal gradient method in our algorithm \textsc{nesta}-\textsc{lasso} and shown that updating the prox-center leads to improved results. Through our experiments, we have shown that using \textsc{nesta}-\textsc{lasso} in the Pareto root-finding method leads to results comparable to those of currently available state-of-the-art methods. Moreover, \textsc{parnes} performs consistently well in all our experiments.


\begin{acknowledgements}
We would like to give special thanks to Emmanuel Cand\`{e}s for helpful
discussions and ideas. The numerical experiments in this paper rely on the
shell scripts and \textsc{matlab}
codes\footnote{\url{http://www.acm.caltech.edu/~nesta/NESTA_ExperimentPackage.zip}}
of J\'{e}r\^{o}me Bobin. We have also benefited from Michael Friedlander and
Ewout van den Berg's \textsc{matlab}
codes\footnote{\url{http://www.cs.ubc.ca/labs/scl/spgl1}} for \textsc{spgl1}.
We are grateful to them for generously making their codes available on the web. Lastly, we would like to thank Dan Gardiner for pointing out the results in \cite{measure0} which we use in the proof of Theorem \ref{remark:direction}. 
\end{acknowledgements}

%
%

\end{document}